\documentclass[12pt]{article}
\usepackage[T1]{fontenc}
\usepackage{amsmath}
\usepackage{amssymb}
\usepackage{amsthm}
\usepackage{tikz}
\usepackage{setspace}
\usetikzlibrary{patterns}

\newtheorem{theorem}{Theorem}[section]

\newtheorem{lemma}[theorem]{Lemma}
\newtheorem{proposition}[theorem]{Proposition}
\newtheorem{corollary}[theorem]{Corollary}

\def\fl#1{\smash{\mathop{\hbox to 11mm{ \rightarrowfill\ }}\limits^{\textstyle #1}}}

\newcommand{\proofof}[1]{\noindent{{\em Proof of #1.}}}
\newcommand{\fproofof}[1]{\noindent{{\em First proof of #1.}}}
\newcommand{\sproofof}[1]{\noindent{{\em Second proof of #1.}}}
\newcommand{\fin}{\hfill$\square$\bigskip}
\newcommand{\tr}[1]{\hspace{-0.5ex}\ ^t\hspace{-0.25ex}#1}
\DeclareMathOperator{\Tor}{Tor}

\newcommand{\Z}{\mathbb{Z}}
\newcommand{\Q}{\mathbb{Q}}

\newcommand{\M}{\mathcal{M}}
\newcommand{\iso}{\fl{\scriptstyle{\cong}}}
\renewcommand{\b}[1]{\boldsymbol{#1}}

\title{Realizing isomorphisms between first homology groups of closed 3-manifolds by borromean surgeries}
\author{Delphine Moussard\footnote{The author was supported by the Italian FIRB project 
"Geometry and topology of low-dimensional manifolds", RBFR10GHHH.}}
\date{}

\begin{document}

\maketitle

\begin{abstract}
We refine Matveev's result asserting that any two closed oriented 3-manifolds can be related by a sequence of 
borromean surgeries if and only if they have isomorphic first homology groups and linking pairings. Indeed, a borromean 
surgery induces a canonical isomorphism between the first homology groups of the involved 3-manifolds, 
which preserves the linking pairing. We prove that any such isomorphism is induced by a sequence of borromean surgeries. 
As an intermediate result, we prove that a given algebraic square finite presentation of the first homology group of a 3-manifold, 
which encodes the linking pairing, can always be obtained from a surgery presentation of the manifold. 
\vspace{1ex}


\end{abstract}

{\small
\begin{spacing}{0.1}
\tableofcontents
\end{spacing}}

    \section*{Introduction}

A well-known result of Matveev \cite[Theorem 2]{Mat} asserts that any two closed connected oriented 3-manifolds can be related by a sequence of 
borromean surgeries if and only if they have isomorphic first homology groups and linking pairings. It is an important 
result, which is useful is particular in the Goussarov-Habiro theory of finite type invariants of 3-manifolds, based 
on borromean surgeries. The more direct application of Matveev's result in this theory is the fact that the degree 0 
invariants are exactly encoded by the isomorphism class of the first homology group equipped with the linking pairing. 

One can study 3-manifolds equipped with an additional structure, as a homological parametrization (to study the Reidemeister 
torsion, see \cite{Mas2}) or a fixed cohomology class (to define Costantino-Geer-Patureau's invariants \cite{CGP}). In these cases, 
a theory of finite type invariants similar to the Goussarov-Habiro's one can be defined. To study such theories, 
it would be useful to have a refined version of Matveev's result. It is the goal of this article to state and prove such a refinement. 
More precisely, a borromean surgery induces a canonical isomorphism between the first homology groups of the involved 3-manifolds, 
which preserves the linking pairing. We prove that any such isomorphism is induced by a sequence of borromean surgeries. 

\paragraph{Historical overview:}
Borromean surgeries were introduced by Matveev in \cite{Mat}. The formalism used in this article is due to Goussarov \cite{Gou}. 
An exposition of Matveev's and Goussarov's definitions, and of the relation between them, is given in \cite[\S1]{Mas3}. 
An equivalent move has been defined by Habiro \cite{Hab}. 

Let us recall the main lines of the proof of the mentioned result of Matveev. The fact that a borromean surgery preserves the isomorphism 
class of the first homology group equipped with the linking pairing is easy, and we focuse on the converse. 
Consider two closed connected oriented 3-manifolds with isomorphic first homology groups and linking pairings, and present them 
by surgery links in $S^3$. When the associated linking matrices are stably equivalent, {\em i.e.} are related 
by stabilizations/destabilizations and unimodular congruences, Matveev proves by topological manipulations 
that the two surgery links are related by borromean surgeries \cite[Lemma 2]{Mat}. This can also be achieved by applying results 
of Murakami-Nakanishi \cite{MuNa} and Garoufalidis-Goussarov-Polyak \cite{GGP} (see Section \ref{secth}, Proposition \ref{propMuNa}). 
To conclude, one needs to prove that any two matrices which define isomorphic groups and pairings are stably equivalent. 
This was first proved by Kneser and Puppe \cite[Satz 3]{KnePu} in the case of a finite group of odd order, and then 
by Durfee \cite[Corollary 4.2]{Dur} for any finite group, with a different and more direct method. 
The case of an infinite group deduces easily (see Kyle \cite[Lemma 1]{Kyle}).

In \cite{Mas3}, Massuyeau gives the analogue of Matveev's result in the case of 3-manifolds with spin structure, proving that two such manifolds 
are related by spin borromean surgeries if and only if they have isomorphic first homology groups and linking pairings, and equal Rochlin invariants 
modulo 8.

\paragraph{Plan of the paper:}
In the first section, we recall the definitions of the linking pairing and of borromean surgeries. 
In Section \ref{secth}, we state the main theorem, and we reduce its proof to another result, 
namely the fact that a given square finite presentation of the first homology group of a closed connected oriented 3-manifold, 
which encodes the linking pairing, can always be obtained from a surgery presentation of the manifold. The next three sections 
are devoted to the proof of this latter theorem. We first review Kirby calculus in Section \ref{secKirby}, and we give 
two different proofs of the second theorem in Sections~\ref{sectop} and~\ref{secalg}. The first proof is of topological nature and quite direct.
The second proof is of algebraic nature; it is more technical, but also more constructive. 
We end with an example in Section \ref{secex}.

\paragraph{Convention:} Unless otherwise mentioned, a 3-manifold is closed, connected and oriented. The homology class of a curve 
$\gamma$ in a manifold is denoted by $[\gamma]$. If $L\subset S^3$ is a link, $S^3(L)$ is the manifold obtained by surgery on $L$. 
If $K$ is a knot in a 3-manifold, a {\em meridian} of $K$ is the boundary of an embedded disk which meets $K$ exactly once. 
If $\Sigma$ and $\gamma$ are respectively a surface and a curve embedded in a 3-manifold, which intersect transversely, then $\langle \Sigma,\gamma\rangle$ 
denotes their algebraic intersection number. 

\paragraph{Acknowledgments:} I wish to thank Christine Lescop for interesting suggestion.

    \section{Preliminaries} \label{secpre}

\paragraph{Linking pairings.}
Let $M$ be a 3-manifold. If $K$ and $J$ are knots in $M$ whose homology classes in $H_1(M;\Z)$ have finite orders, 
one can define the linking number of $K$ and $J$ as follows. Let $d$ be the order of $[K]\in H_1(M;\Z)$, and let $T(K)$ be a tubular neighborhood of $K$. 
A {\em Seifert surface} of $K$ is a surface $\Sigma$ embedded in $M\setminus Int(T(K))$ whose boundary $\partial \Sigma \subset \partial T(K)$ 
satisfies $[\partial\Sigma]=d[K]$ in $H_1(T(K);\Z)$. Such a surface always exists and can be chosen transverse to $J$. Define the linking number of $K$ and $J$ 
as $lk(K,J)=\frac{1}{d}\langle\Sigma,J\rangle$. It does not depend on the choice of the surface $\Sigma$, and it is symmetric. 
If the knots $K$ and $J$ are modified within their respective homology classes, the value of the linking number is modified by an integer. 
This allows the following definition. The {\em linking pairing} of M is the $\Z$-bilinear form
$$\begin{array}{c c c c } 
               \varphi_M : & \Tor(H_1(M;\Z))\times\Tor(H_1(M;\Z)) & \to & \Q/\Z \\
               & ([K],[J]) & \mapsto & lk(K,J)\ mod\ \Z
              \end{array}$$
where $\Tor$ stands for the torsion subgroup. This form is symmetric and non-degenerate. 

\paragraph{Framings.}
A 1-dimensional object is {\em framed} if it is equipped with a fixed normal vector field. In the case of a knot $K$ in a 3-manifold, it is equivalent to fix 
a {\em parallel} of the knot, {\em i.e.} a simple closed curve $\ell(K)$ on the boundary of a tubular neighborhood $T(K)$ of $K$ which is isotopic to $K$ 
in $T(K)$. If $K$ is a framed knot in a 3-manifold, with fixed parallel $\ell(K)$, whose homology class has finite order, define the self-linking of $K$ 
by $lk(K,K)=lk(K,\ell(K))$.

\paragraph{Borromean surgeries.}
The {\em standard Y-graph} is the graph $\Gamma_0\subset \mathbb{R}^2$ represented in Figure \ref{figY}. 
\begin{figure}[htb] 
\begin{center}
\begin{tikzpicture} [scale=0.15]
\newcommand{\feuille}[1]{
\draw[rotate=#1,thick,color=gray] (0,-11) circle (5);
\draw[rotate=#1,thick,color=gray] (0,-11) circle (1);
\draw[rotate=#1,line width=8pt,color=white] (-2,-6.42) -- (2,-6.42);
\draw[rotate=#1,thick,color=gray] (2,-1.15) -- (2,-6.42);
\draw[rotate=#1,thick,color=gray] (-2,-1.15) -- (-2,-6.42);
\draw[white,rotate=#1,line width=5pt] (0,0) -- (0,-8);
\draw[rotate=#1] (0,0) -- (0,-8);
\draw[rotate=#1] (0,-11) circle (3);}
\feuille{0}
\feuille{120}
\feuille{-120}
\draw[->] (15,1) -- (12,3);
\draw (18,0) node {leaf};
\draw (4.7,-9.2) node{$\Gamma_0$};
\draw[color=gray] (6.5,-16.8) node{$\Sigma(\Gamma_0)$};
\end{tikzpicture}
\end{center}
\caption{The standard Y-graph}\label{figY}
\end{figure}
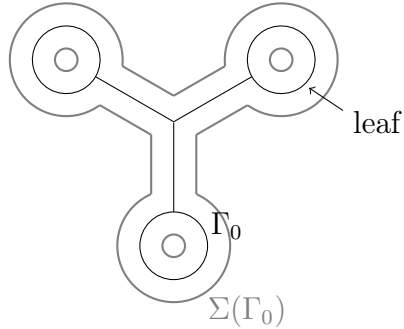
With $\Gamma_0$ is associated a regular neighborhood $\Sigma(\Gamma_0)$ of $\Gamma_0$ in the plane. 
The surface $\Sigma(\Gamma_0)$ is oriented with the usual convention. 
Let $M$ be a 3-manifold and let $h:\Sigma(\Gamma_0)\to M$ be an embedding. The image $\Gamma$ of $\Gamma_0$ is a {\em Y-graph}, endowed with  
its {\em associated surface} $\Sigma(\Gamma)=h(\Sigma(\Gamma_0))$. The looped edges of $\Gamma$ are its {\em leaves}. 
The Y-graph $\Gamma$ is equipped with the framing induced by $\Sigma(\Gamma)$. 

\begin{figure}[hbt] 
\begin{center}
\begin{tikzpicture} [scale=0.15]
\begin{scope}
\newcommand{\feuille}[1]{
\draw[rotate=#1] (0,0) -- (0,-8);
\draw[rotate=#1] (0,-11) circle (3);}
\feuille{0}
\feuille{120}
\feuille{-120}
\draw (3,-4) node{$\Gamma$};
\end{scope}
\draw[very thick,->] (21.5,-3) -- (23.5,-3);
\begin{scope}[xshift=1200]
\newcommand{\bras}[1]{
\draw[rotate=#1] (0,-1.5) circle (2.5);
\draw [rotate=#1,white,line width=8pt] (-0.95,-4) -- (0.95,-4);
\draw[rotate=#1] {(0,-11) circle (3) (1,-3.9) -- (1,-7.6)};
\draw[rotate=#1,white,line width=6pt] (-1,-5) -- (-1,-8.7);
\draw[rotate=#1] {(-1,-3.9) -- (-1,-8.7) (-1,-8.7) arc (-180:0:1)};}
\bras{0}
\draw [white,line width=6pt,rotate=120] (0,-1.5) circle (2.5);
\bras{120}
\draw [rotate=-120,white,line width=6pt] (-1.77,0.27) arc (135:190:2.5);
\draw [rotate=-120,white,line width=6pt] (1.77,0.27) arc (45:90:2.5);
\bras{-120}
\draw [white,line width=6pt] (-1.77,0.27) arc (135:190:2.5);
\draw [white,line width=6pt] (1.77,0.27) arc (45:90:2.5);
\draw (-1.77,0.27) arc (135:190:2.5);
\draw (1.77,0.27) arc (45:90:2.5);
\draw (3.5,-4.5) node{$L$};
\end{scope}
\end{tikzpicture}
\end{center}
\caption{Y-graph and associated surgery link}\label{figborro}
\end{figure}
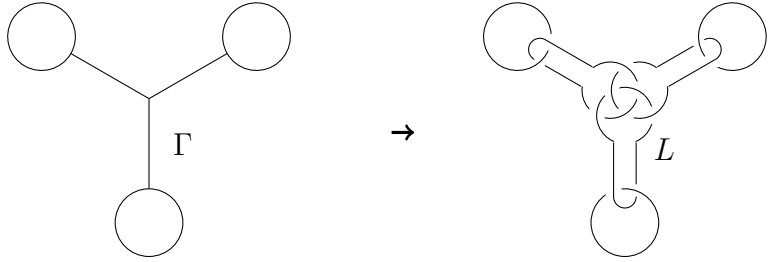

Let $\Gamma$ be a Y-graph in a 3-manifold $M$ (which may have a non-empty boundary). Let $\Sigma(\Gamma)$ be its associated surface. 
In $\Sigma(\Gamma)\times[-1,1]$, associate with $\Gamma$ the six-component link $L$ represented in Figure \ref{figborro}, 
with the blackboard framing. The \emph{borromean surgery on $\Gamma$} is the usual surgery along the framed link $L$. 
The manifold obtained from $M$ by surgery on $\Gamma$ is denoted by $M(\Gamma)$. 

The borromean surgery on a Y-graph $\Gamma$ in a 3-manifold $M$ can be realized by removing the interior of a tubular neighborhood 
$N$ of $\Gamma$ and gluing instead another genus 3 handlebody, via an isomorphism of their boundaries which is the identity in homology 
(\cite[Section 6]{Mat}, see also \cite[Lemma 1]{Mas3}). This implies that 
$\ker(H_1(\partial N;\Z)\to H_1(N;\Z))=\ker(H_1(\partial N;\Z)\to H_1(N(\Gamma);Z))\subset H_1(\partial N;\Z)$, where we consider the applications 
induced in homology by the natural inclusions.

A {\em Y-link} in a 3-manifold is a disjoint union of Y-graphs. The {\em borromean surgery on a Y-link} is given by the simultaneous 
borromean surgeries on all its components. Note that a finite sequence of borromean surgeries can always be performed by a single 
borromean surgery on a Y-link.

    \section{Realizing isomorphisms beetween linking pairings} \label{secth}

We begin with the lemma whose converse is the object of the article.
\begin{lemma} \label{lemmaisoind}
 Let $M$ be a 3-manifold. Let $\Gamma$ be a Y-link in $M$. The surgery on $\Gamma$ induces a canonical isomorphism 
$\xi_\Gamma : H_1(M;\Z) \iso H_1(M(\Gamma);\Z)$ which preserves the linking pairing.
\end{lemma}
\begin{proof}
 Let $N$ be a tubular neighborhood of $\Gamma$ in $M$, and set $X=M\setminus N$. The Mayer-Vietoris sequence associated with $M=X\cup N$ 
yields the exact sequence:
$$H_1(\partial N;\Z)\to H_1(N;\Z)\oplus H_1(X;\Z)\to H_1(M;\Z)\to 0.$$
Since $H_1(\partial N;\Z)\cong H\oplus H_1(N;\Z)$, where $H=\ker(H_1(\partial N;\Z)\to H_1(N;\Z))$, we have $\displaystyle H_1(M;\Z)\cong \frac{H_1(X;\Z)}{H}$. 
We have seen in the previous section that $H=\ker(H_1(\partial N;\Z)\to H_1(N(\Gamma);Z))\subset H_1(\partial N;\Z)$, 
hence we have similarly $\displaystyle H_1(M(\Gamma);\Z)\cong \frac{H_1(X;\Z)}{H}$, and it follows that $H_1(M;\Z)$ and $H_1(M(\Gamma);\Z)$ are canonically identified. 

The canonical isomorphism $\xi_{\Gamma}:H_1(M;\Z)\to H_1(M(\Gamma);\Z)$ can be described as follows. Let $\eta\in H_1(M;\Z)$. 
Represent $\eta$ by a knot $K$ in $M$ disjoint from $\Gamma$. The knot $K$ is not affected by the surgery. The image $\xi_{\Gamma}(\eta)$ 
is the homology class of $K\subset M(\Gamma)$. 

Let us check that the linking numbers are preserved. Let $K$ and $J$ be knots in $M$, disjoint from $\Gamma$, whose homology classes have finite orders. 
Let $\Sigma$ be a Seifert surface of $K$, transverse to $\Gamma$ and $J$. We may assume that the only edges of $\Gamma$ 
that $\Sigma$ meets are its leaves. The surgery modifies the tubular neighborhood $N$ of $\Gamma$. 
At each point of $\Sigma\cap\Gamma$, remove a little disk from $\Sigma$ and replace it, after surgery, with the surface drawn 
in Figure \ref{figsurface}, where the apparent boundary inside $N(\Gamma)$ bounds a disk in the corresponding reglued torus. 
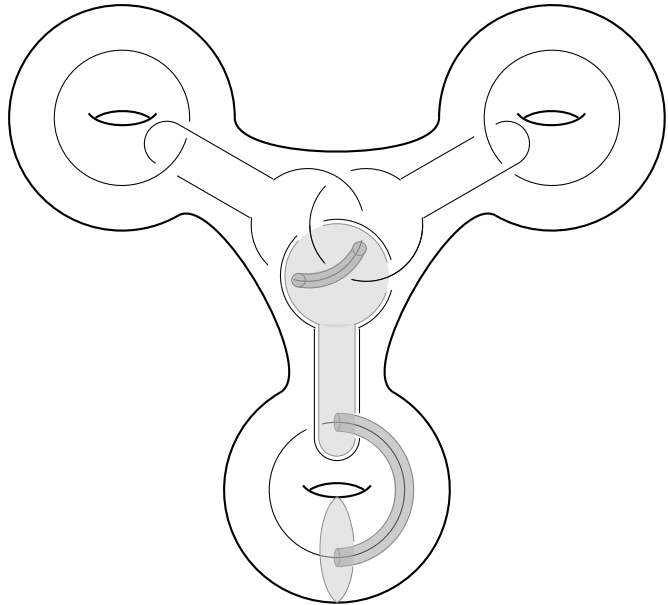
\begin{figure}[htb] 
\begin{center}
\begin{tikzpicture} [scale=0.3,fill opacity=0.5]

\newcommand{\creux}[2]{
\draw[xshift=#1,yshift=#2,thick] (-1.5,0.2) ..controls +(0.5,-0.7) and +(-0.5,-0.7) .. (1.5,0.2);
\draw[xshift=#1,yshift=#2,thick] (-1.2,0) ..controls +(0.6,0.4) and +(-0.6,0.4) .. (1.2,0);}
\creux{0}{-11cm}
\creux{9.5cm}{5.5cm}
\creux{-9.5cm}{5.5cm}

\newcommand{\bras}[1]{
\draw[rotate=#1,thick] (-2.5,-6.67) arc (-240:60:5);
\draw[rotate=#1] (0,-1.5) circle (2.5);
\draw [rotate=#1,white,line width=8pt] (-0.95,-4) -- (0.95,-4);
\draw[rotate=#1] {(0,-11) circle (3) (1,-3.85) -- (1,-7.6)};
\draw[rotate=#1,white,line width=6pt] (-1,-5) -- (-1,-8.7);
\draw[rotate=#1] {(-1,-3.85) -- (-1,-8.7) (-1,-8.7) arc (-180:0:1)};
\draw[rotate=#1,thick] (2.5,-6.67) ..controls +(-1.73,1) and +(-1.73,1) .. (7.03,1.17);}
\bras{0}
\draw [white,line width=10pt,rotate=120] (0,-1.5) circle (2.5);
\bras{120}
\draw [rotate=-120,white,line width=10pt] (-1.77,0.27) arc (135:190:2.5);
\draw [rotate=-120,white,line width=10pt] (1.77,0.27) arc (45:90:2.5);
\bras{-120}
\draw [white,line width=10pt] (-1.77,0.27) arc (135:190:2.5);
\draw [white,line width=10pt] (1.77,0.27) arc (45:90:2.5);
\draw (-1.77,0.27) arc (135:190:2.5);
\draw (1.77,0.27) arc (45:90:2.5);

\draw[gray] (0,-11.3) ..controls +(-1,-0.8) and +(-1,0.8) .. (0,-16);
\draw[gray] (0,-11.3) ..controls +(0.5,-0.3) and +(0,0.4) .. (0.75,-13.3) (0.7,-14.4) .. controls +(0,-0.3) and +(0.5,0.3) .. (0,-16);
\draw[gray] (0,-13.6) ..controls +(-0.17,-0.13) and +(-0.17,0.13) .. (0,-14.4);
\draw[gray] (0,-13.6) ..controls +(0.17,-0.13) and +(0.17,0.13) .. (0,-14.4);
\draw[gray] (0,-7.6) ..controls +(-0.17,-0.13) and +(-0.17,0.13) .. (0,-8.4);
\draw[gray] (0,-7.6) ..controls +(0.17,-0.13) and +(0.17,0.13) .. (0,-8.4);
\draw[gray] (0,-13.6) arc (-90:90:2.6);
\draw[gray] (0,-14.4) arc (-90:90:3.4);
\draw[gray] (0,-1.5) circle (2.3);
\draw [rotate=120,white,line width=10pt] (-1.77,0.27) arc (135:190:2.5);
\draw [rotate=120,white,line width=10pt] (1.77,0.27) arc (45:90:2.5);
\draw [rotate=120] (-1.77,0.27) arc (135:190:2.5);
\draw [rotate=120] (1.77,0.27) arc (45:90:2.5);
\draw [white,line width=8pt] (-0.75,-3.8) -- (0.75,-3.8);
\draw[gray] (0.8,-3.65) -- (0.8,-7.6) (-0.8,-3.65) -- (-0.8,-8.7) (-0.8,-8.7) arc (-180:0:0.8);
\fill[gray!40] (0,-1.5) circle (2.3);
\fill[gray!40] (-0.8,-3.65) -- (-0.8,-8.7) arc (-180:0:0.8) -- (0.8,-3.65) -- (-0.8,-3.65);
\draw[gray] (-1.7,-1.7) circle (0.3);
\draw[gray] (1,-0.3) circle (0.3);
\draw[gray] (-1.7,-1.4) arc (-100:-30:2.3);
\draw[gray] (-1.7,-2) arc (-100:-25:2.7);
\fill[gray!80] (-1.7,-2) arc (-100:-25:2.7) (1.26,-0.45) ..controls +(0.15,0.26) and +(0.15,0.26) .. (0.8,-0.1) 
  ..controls +(-0.6,-1.02) and +(1.2,0) .. (-1.7,-1.4) ..controls +(-0.35,0) and +(-0.35,0) .. (-1.7,-2);
\fill[gray!40] (0,-11.3) ..controls +(-1,-0.8) and +(-1,0.8) .. (0,-16) .. controls +(1,0.8) and +(1,-0.8) .. (0,-11.3);
\fill[gray!80] (0,-13.6) ..controls +(-0.17,-0.13) and +(-0.17,0.13) .. (0,-14.4) (0,-14.4) arc (-90:90:3.4) 
  (0,-7.6) ..controls +(-0.17,-0.13) and +(-0.17,0.13) .. (0,-8.4) ..controls +(3.4,0) and +(3.4,0) .. (0,-13.6);

\draw[very thin] (3.8,0.75) arc (0:-105:2.5);
\draw[very thin] (-1.2,0.75) arc (-180:-138:2.5);

\end{tikzpicture}
\end{center}
\caption{Surface in the reglued handlebody} \label{figsurface}
\end{figure}
This provides a surface $\Sigma'$ in $M(\Gamma)$, which is again a Seifert surface of $K$, and such that the algebraic intersection numbers $\langle \Sigma,J\rangle$ and 
$\langle \Sigma',J\rangle$ are equal.
\end{proof}

The main result of the article is the following converse of Lemma \ref{lemmaisoind}.
\begin{theorem} \label{thmain}
 Let $M$ and $N$ be 3-manifolds. Let $\xi:H_1(M;\Z)\iso H_1(N;\Z)$ be an isomorphism which preserves the linking pairing. 
Then there is a Y-link $\Gamma$ in $M$ such that $N\cong M(\Gamma)$ and $\xi=\xi_\Gamma$.
\end{theorem}

The proof of Theorem \ref{thmain} uses surgery presentations of the manifolds. In order to have an associated presentation of the first homology group , 
we need to fix an orientation of the surgery link.

An {\em oriented surgery presentation} of a 3-manifold $M$ is an oriented framed link 
$L\subset S^3$ such that $M$ is obtained from $S^3$ by surgery on $L$, with a given numbering $L=\sqcup_{1\leq i\leq n}L_i$ 
of the knot components of $L$. The {\em linking matrix of $L$} is the matrix $A(L)$ defined by $(A(L))_{ij}=-lk(L_i,L_j)$ 
(this convention with a minus sign is unusual, but simplifies the expression of the linking pairing 
in terms of the linking matrix). For $1\leq i\leq n$, let $m_i$ be an {\em oriented meridian} 
of $L_i$, {\em i.e.} such that $lk(L_i,m_i)=1$. The {\em presentation of $H_1(M;\Z)$ induced by the oriented surgery presentation $L$} 
is given by the family of generators $\b{m}=([m_1],..,[m_n])$ up to the relations given by the columns of $A(L)$. 
It is well known that any 3-manifold admits an (oriented) surgery presentation (Lickorish \cite{Lick2}, Wallace \cite{Wal}).
\begin{proposition} \label{propMuNa}
 Let $M$ (resp. $M'$) be a 3-manifold with oriented surgery presentation $L$ (resp. $L'$). Denote by $\b{m}$ (resp. $\b{m'}$) 
the associated family of generators of $H_1(M;\Z)$ (resp. $H_1(M';\Z)$). If $A(L)=A(L')$, then there is a Y-link $\Gamma$ in $M$ 
such that $M'\cong M(\Gamma)$ and $\xi_\Gamma(\b{m})=\b{m'}$.
\end{proposition}
\begin{proof}
 A {\em $\Delta$-move} on an oriented framed link is a local move as represented in Figure~\ref{figdeltamove} involving three components of the link 
(with possible repetitions), which keeps unchanged the orientations and the framings. 
By \cite[Theorem 1.1]{MuNa}, 
\newcommand{\DM}{
\draw (-3,-1) -- (3,-1);
\draw[color=white,line width=8pt,rotate=120] (-3,-1) -- (0,-1);
\draw[rotate=120] (-3,-1) -- (3,-1);
\draw[color=white,line width=8pt,rotate=240] (-3,-1) -- (0,-1);
\draw[rotate=240] (-3,-1) -- (3,-1);
\draw[color=white,line width=8pt] (-3,-1) -- (0,-1);
\draw (-3,-1) -- (1,-1);}
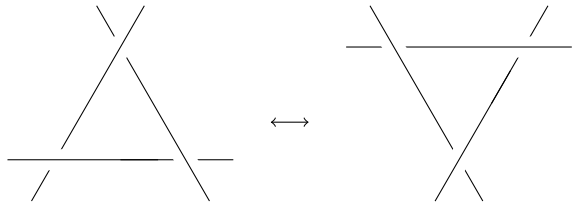
\begin{figure}[htb] 
\begin{center}
\begin{tikzpicture} [scale=0.5]
\DM
\draw[<->] (4,0) -- (5,0);
\begin{scope} [xshift=9cm,yshift=1cm,rotate=60]
\DM
\end{scope}
\end{tikzpicture}
\end{center}
\caption{$\Delta$-move} \label{figdeltamove}
\end{figure}
any two oriented framed links with the same linking matrix are related by a sequence of $\Delta$-moves. 
By \cite[Lemma 2.1]{GGP}, such a $\Delta$-move can be realised by a borromean surgery (see Figure~\ref{figdeltasur}), 
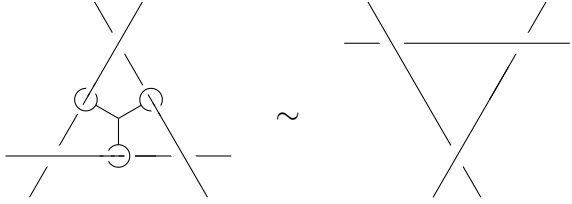
\begin{figure}[htb] 
\begin{center}
\begin{tikzpicture} [scale=0.5]
\DM
\newcommand{\leaf}[1]{\begin{scope} [rotate=#1]
\draw (0,0) -- (0,-0.7);
\draw[color=white,line width=4pt] (0.3,-0.9) -- (0.3,-1.1);
\draw (0,-1) circle (0.3);
\draw[color=white,line width=4pt] (-0.4,-1) -- (-0.2,-1);
\draw (-0.5,-1) -- (-0.1,-1);
\end{scope}}
\leaf{0}
\leaf{120}
\leaf{240}
\draw[line width=1.5pt] (4.5,0) node {$\sim$};
\begin{scope} [xshift=9cm,yshift=1cm,rotate=60]
\DM
\end{scope}
\end{tikzpicture}
\end{center}
\caption{Borromean surgery realizing a $\Delta$-move} \label{figdeltasur}
\end{figure}
which means that the pair (3-manifold, link) obtained by the surgery is homeomorphic to the pair ($S^3$, link) obtained by the $\Delta$-move. 

Finally, there exists a Y-link $\Gamma$ in $S^3\setminus L$ 
such that $S^3(\Gamma)\cong S^3$, and the copy of $L$ in $S^3(\Gamma)$ is isotopic to $L'$. Since the surgery on $\Gamma$ is performed 
in the complement of $L$, we can consider the copy of $\Gamma$ in $M$, and we have $M(\Gamma)\cong M'$. The assertion on $\xi_\Gamma$ 
follows from the fact that the meridians of the components of $L$ are not affected by the $\Delta$-moves. 
\end{proof}

We shall prove that we can choose a surgery presentation providing a fixed presentation of the first homology group. 
Let us fix a few algebraic definitions. 

A {\em linking matrix} is a square symmetric matrix with integral coefficients. A linking matrix $A$ is {\em non-degenerate} 
if $\det(A)\neq0$, and {\em admissible} if $A=\begin{pmatrix} A_0 & 0 \\ 0 & 0 \end{pmatrix}$ where $A_0$ is a non-degenerate linking 
matrix. By \cite[Lemma 1]{Kyle}, any linking matrix is congruent over the integers to an admissible linking matrix.

Let $A$ be a linking matrix. If $B=\begin{pmatrix} D & 0 \\ 0 & A \end{pmatrix}$, where $D$ is a diagonal matrix whose diagonal terms are $\pm1$, 
the matrix $B$ is a {\em stabilization} of $A$ and $A$ is a {\em destabilization} of $B$.

A {\em linking pairing} on a finite abelian group $H$ is a symmetric bilinear pairing $\varphi:H\times H\to\Q/\Z$. 
A {\em linked group} is a finitely generated abelian group $H$ whose torsion subgroup $\Tor(H)$ is equipped with a linking pairing. 
Let $A$ be a linking matrix of size $n$ and rank $r$. 
The matrix $A$ can be written $A=\tr{P}\begin{pmatrix} A_0 & 0 \\ 0 & 0 \end{pmatrix}P$, where $A_0$ is a non-degenerate linking matrix, 
$P$ is an integral matrix invertible over the integers, and $\tr{P}$ is the transpose of $P$. 
An {\em $A$-presentation} of a linked group $H$ is a family $\b{\eta}=(\eta_1,..,\eta_n)\subset H$ such that:
\begin{itemize}
 \item the following sequence is exact, where the first map is given by $A$ in the canonical bases and the second map sends 
  the canonical basis on $\b{\eta}$, $$\Z^n\fl{A}\Z^n\fl{ } H\fl{ } 0$$
 \item $\varphi((P^{-1}\eta)_i,(P^{-1}\eta)_j)=(A_0^{-1})_{ij}\ mod\ \Z$ for $1\leq i,j\leq r$, where $\varphi$ is the linking pairing 
  on $\Tor(H)$.
\end{itemize}

Let $M$ be a 3-manifold. The group $H_1(M;\Z)$ has a natural structure of linked group given by the linking pairing 
$\varphi_M$
of $M$. If $L$ is an oriented surgery presentation of $M$, 
the presentation of $H_1(M;\Z)$ induced by $L$ is an $A(L)$-presentation. 

\begin{theorem} \label{thpres}
 Let $M$ be a 3-manifold. Let $A$ be a linking matrix. Assume there is an $A$-presentation $\b{\eta}=(\eta_1,..,\eta_n)$ 
of $H_1(M;\Z)$. Then there is an oriented surgery presentation $L$ of $M$ such that $A(L)$ is a stabilization of $A$, 
and $L$ induces the $A(L)$-presentation $\b{\eta'}=(0,..,0,\eta_1,..,\eta_n)$.
\end{theorem}
This result is the purpose of the next three sections. We first review Kirby calculus in Section~\ref{secKirby},
and we give two independant proofs of the theorem in Sections~\ref{sectop} and~\ref{secalg}.

\ \\
\proofof{Theorem \ref{thmain}}
Let $L_1$ be an oriented surgery presentation of $M$. 
Let $\b{\eta}$ be the associated $A(L_1)$-presentation of $H_1(M;\Z)$. The family $\b{\nu}=(\xi(\eta_1),..,\xi(\eta_n))$ 
is an $A(L_1)$-presentation of $H_1(N;\Z)$. By Theorem \ref{thpres}, there is an oriented surgery presentation $L_2$ of $N$ 
such that $A(L_2)$ is a stabilization of $A(L_1)$, and $L_2$ induces the $A(L_2)$-presentation $\b{\nu'}=(0,..,0,\xi(\eta_1),..,\xi(\eta_n))$. 
Define a surgery link $L_1'$ by adding $(\pm1)$-framed trivial components to $L_1$, and renumber the components, so that $A(L_1')=A(L_2)$ 
and $L_1'$ induces the $A(L_2)$-presentation $\b{\eta'}=(0,..,0,\eta_1,..,\eta_n)$. Conclude with Proposition \ref{propMuNa}.
\fin

    \section{Kirby transformations} \label{secKirby}

By a well-known theorem of Kirby \cite{Kir}, any two oriented surgery presentations of a 3-manifold $M$ 
are related by a finite sequence of the following K1 and K2 moves.
\begin{itemize}
 \item K1 (stabilization/destabilization): add or remove a $(\pm1)$-framed trivial component, unlinked with the other components, to the surgery link.
 \item K2: add or substract a component of the surgery link to another (see Figure \ref{figK2}).
\end{itemize}
\begin{figure}[htb] 
\begin{center}
\begin{tikzpicture} [scale=0.3]
\draw (0,0) -- (0,-6);
\draw[dashed] (0,2) -- (0,0) (0,-6) -- (0,-8);
\draw (-4,-4.5) arc (-90:90:1.5);
\draw[dashed] (-4,-1.5) arc (90:270:1.5);
\draw[->] (-3.9,-1.5) -- (-4,-1.5);
\draw (-4,-0.5) node{$L_j$};
\draw[->] (0,0) -- (0,-1);
\draw (1,-1) node{$L_i$};
\draw[->,line width=1.5pt] (6.5,-3) -- (7.5,-3);
\begin{scope} [xshift=18cm]
\draw (0,0) -- (0,-2.5) (0,-3.5) -- (0,-6);
\draw[dashed] (0,2) -- (0,0) (0,-6) -- (0,-8);
\draw (-4,-4.5) arc (-90:90:1.5);
\draw[dashed] (-4,-1.5) arc (90:270:1.5);
\draw (-4,-4.8) arc (-90:90:1.8);
\draw[dashed] (-4,-1.2) arc (90:270:1.8);
\draw[color=white,line width=3pt] (-2.26,-2.53) -- (-2.26,-3.48);
\draw (-2.26,-2.5) -- (0,-2.5) (-2.26,-3.5) -- (0,-3.5);
\draw[->] (-3.9,-1.5) -- (-4,-1.5);
\draw[->] (0,0) -- (0,-1);
\draw (1,-1) node{$L_i'$};
\end{scope}
\end{tikzpicture}
\caption{K2 move} \label{figK2}
\end{center}
\end{figure}

The effect of a Kirby move on the linking matrix is a stabilization/destabilization in the case of a K1 move, and a unimodular congruence 
in the case of a K2 move (where {\em unimodular} means that the congruence matrix has determinant 1). 
Conversely, a stabilization or a unimodular congruence on the linking matrix can be realized by a finite sequence 
of Kirby moves. Note that this converse does not hold for a destabilization. In order to prove Theorem \ref{thpres}, we need 
to understand the effect of a Kirby move, performed on a surgery link $L$, on the associated $A(L)$-presentation of the first homology group. 
The case of a K1 move is obvious: we add or remove a trivial generator to the presentation. Let us consider the case of a K2 move. 

Let $M$ be a 3-manifold with oriented surgery presentation $L=\sqcup_{1\leq k\leq n}L_k$. Define another oriented surgery presentation 
$L'=\sqcup_{1\leq k\leq n}L_k'$ of $M$ by adding the component $L_j$ to $L_i$ (Fig.~\ref{figK2}). The reason why $L'$ is again 
a surgery presentation of $M$ is the following. We can perform first the surgery on the component $L_j$ and then switch 
the component $L_i$ along the meridian disk bounded by the parallel of $L_j$. Now the effect on the meridians of the components of $L$ 
is represented in Figure~\ref{figK2mer}. If $m_k$ (resp. $m_k'$) is an oriented meridian of $L_k$ (resp. $L_k'$), 
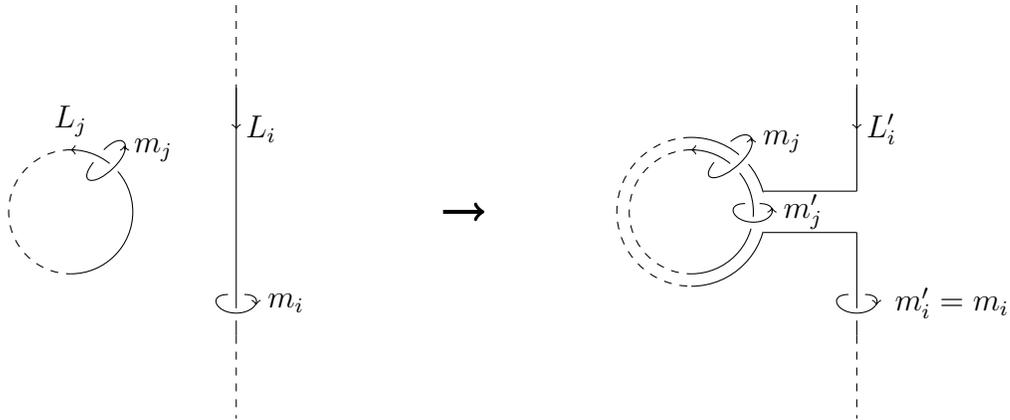
\begin{figure}[htb] 
\begin{center}
\begin{tikzpicture} [scale=0.55]
\draw (0,0) -- (0,-6);
\draw[dashed] (0,2) -- (0,0) (0,-6) -- (0,-8);
\draw (-4,-4.5) arc (-90:90:1.5);
\draw[dashed] (-4,-1.5) arc (90:270:1.5);
\draw[->] (-3.9,-1.5) -- (-4,-1.5);
\draw (-4,-0.8) node{$L_j$};
\draw[->] (0,0) -- (0,-1);
\draw (0.6,-1) node{$L_i$};
\draw[white] (-3,-2) node{$\bullet$};
\draw (-3.45,-1.8) .. controls +(-0.4,-0.5) and +(-0.5,-0.5) .. (-2.95,-1.92) .. controls +(0.5,0.5) and +(0.5,0.5) .. (-3.2,-1.5);
\draw[->] (-2.7,-1.5) -- (-2.7,-1.4);
\draw (-2,-1.5) node{$m_j$};
\draw[white] (0,-5.5) node{$\bullet$};
\draw (-0.2,-5) .. controls +(-0.5,0) and +(-0.5,-0) .. (0,-5.45) .. controls +(0.5,0) and +(0.5,0) .. (0.2,-5);
\draw[->] (0.48,-5.2) -- (0.48,-5.25);
\draw (1.2,-5.2) node{$m_i$};
\draw[->,line width=1.5pt] (5,-3) -- (6,-3);
\begin{scope} [xshift=15cm]
\draw (0,0) -- (0,-2.5) (0,-3.5) -- (0,-6);
\draw[dashed] (0,2) -- (0,0) (0,-6) -- (0,-8);
\draw (-4,-4.5) arc (-90:90:1.5);
\draw[dashed] (-4,-1.5) arc (90:270:1.5);
\draw (-4,-4.8) arc (-90:90:1.8);
\draw[dashed] (-4,-1.2) arc (90:270:1.8);
\draw[color=white,line width=3pt] (-2.26,-2.53) -- (-2.26,-3.48);
\draw (-2.26,-2.5) -- (0,-2.5) (-2.26,-3.5) -- (0,-3.5);
\draw[->] (-3.9,-1.5) -- (-4,-1.5);
\draw[->] (0,0) -- (0,-1);
\draw (0.6,-1) node{$L_i'$};
\draw[white] (0,-5.5) node{$\bullet$};
\draw (-0.2,-5) .. controls +(-0.5,0) and +(-0.5,0) .. (0,-5.45) .. controls +(0.5,0) and +(0.5,0) .. (0.2,-5);
\draw[->] (0.48,-5.2) -- (0.48,-5.25);
\draw (2.3,-5.2) node{$m_i'=m_i$};
\draw[white] (-2.5,-3.3) node{$\bullet$};
\draw (-2.7,-2.8) .. controls +(-0.5,0) and +(-0.5,-0) .. (-2.5,-3.25) .. controls +(0.5,0) and +(0.5,0) .. (-2.35,-2.8);
\draw[->] (-2.05,-3) -- (-2.05,-2.9);
\draw (-1.3,-3) node{$m_j'$};
\draw[color=white,line width=5pt] (-3,-2) -- (-2,-1);
\draw (-3.45,-1.8) .. controls +(-0.4,-0.5) and +(-0.5,-0.5) .. (-2.8,-1.8) .. controls +(0.5,0.5) and +(0.5,0.4) .. (-3.05,-1.3);
\draw[->] (-2.55,-1.3) -- (-2.55,-1.2);
\draw (-1.8,-1.3) node{$m_j$};
\end{scope}
\end{tikzpicture}
\caption{Effect of a K2 move on the meridians} \label{figK2mer}
\end{center}
\end{figure}
then $[m_k']=[m_k]$ for $k\neq j$ and $[m_j']=[m_j]-[m_i]$.
The induced congruence between the linking matrices is $A(L')=T_{ij}A(L)T_{ji}$, where $T_{k\ell}$ is the transvection matrix defined 
as $T_{k\ell}=I+E_{k\ell}$ and $E_{k\ell}$ is the matrix whose only non-trivial coefficient is a 1 at the $k$-th row and $\ell$-th column. 
If $\b{m}$ (resp. $\b{m'}$) is the presentation of $H_1(M;\Z)$ associated with $L$ (resp. $L'$), we have $T_{ji}\b{m'}=\b{m}$. 
If the component $L_j$ is substracted to $L_i$, then the transvection matrix which appears is $T_{ji}^{-1}=I-E_{ji}$. 
Since any {\em unimodular matrix} (square integral matrix of determinant 1) is a product of transvection matrices, 
we obtain the following statement.
\begin{lemma} \label{lemmaKirby}
 Let $M$ be a 3-manifold. Let $L$ be an oriented surgery presentation of $M$. Let $\b{m}$ be the associated presentation 
of $H_1(M;\Z)$. Consider $A=\tr{P}A(L)P$, where $P$ is a unimodular matrix. Then there is an oriented surgery presentation $L'$ of $M$, 
that can be obtained from $L$ by a sequence of K2 moves, such that $A(L')=A$ and the associated presentation $\b{m'}$ of $H_1(M;\Z)$ 
satisfies $P\b{m'}=\b{m}$.
\end{lemma}
\begin{corollary}
 Any 3-manifold admits an oriented surgery presentation whose associated linking matrix is admissible.
\end{corollary}

  \section{Presentations of $H_1(M;\Z)$, topological version} \label{sectop}

An {\em integral homology sphere}, or {\em $\Z$-sphere}, is a 3-manifold which has the same homology as the standard 3-sphere $S^3$. 
Define oriented surgery presentations on links in $\Z$-spheres as in $S^3$. 
The next proposition is an equivalent of Theorem \ref{thpres} for surgery presentations in $\Z$-spheres. 
We will end the section by deducing Theorem \ref{thpres} from this proposition and Matveev's result for $\Z$-spheres.
  
\begin{proposition} \label{propZHS}
 Let $M$ be a 3-manifold. Let $A$ be an admissible linking matrix. Assume there is an $A$-presentation $\b{\eta}=(\eta_1,..,\eta_n)$ 
of $H_1(M;\Z)$. Then there is a $\Z$-sphere $N$ and an oriented surgery presentation $L\subset N$ of $M$ such that $A(L)=A$ 
and $L$ induces the $A$-presentation $\b{\eta}$.
\end{proposition}
\begin{proof}
The idea of this proof is to kill the homology of $M$ by surgeries in order to obtain a $\Z$-sphere, and to consider the inverse surgeries. 

Write $A=\begin{pmatrix} A_0 & 0 \\ 0 & 0 \end{pmatrix}$ where $A_0=(a_{ij})_{1\leq i,j\leq m}$ is a non-degenerate linking matrix.
For $1\leq i\leq n$, let $K_i$ be a framed knot in $M$ whose homology class is $\eta_i$. Choose the $K_i$ pairwise disjoint. 
Let $T(K_i)$ be a tubular neighborhood of $K_i$, let $m(K_i)\subset \partial T(K_i)$ be an oriented meridian of $K_i$, 
and let $\ell(K_i)\subset \partial T(K_i)$ be the fixed parallel of $K_i$. 

For $1\hspace{-1pt}\leq \hspace{-1pt}i,j\hspace{-1pt}\leq \hspace{-1pt}m$, we have $lk(K_i,K_j)\equiv(A_0^{-1})_{ij}\ mod\ \Z$. 
Since adding to a $K_i$ some meridians of the $K_j$ does not modify the homology class of $K_i$ 
in $M$ (and since we can modify the choice of the $\ell(K_i)$), we may assume, for all $1\leq i,j\leq m$, that: 
$$lk(K_i,K_j)=(A_0^{-1})_{ij}.$$
Since the columns of $A$ give relations on the $\eta_i$, there are 2-chains $\Sigma_i$ 
in $X_0=M\setminus \sqcup_{1\leq i\leq m}Int(T(K_i))$ such that 
$$\partial \Sigma_i=\sum_{j=1}^m a_{ij}\ell(K_j)+\sum_{j=1}^m b_{ij}m(K_j),$$  
for some integers $b_{ij}$. 

It follows from Poincar\'e duality that there are closed oriented surfaces $S_i$, for $m+1\leq i\leq n$, such that $$\langle S_i,\eta_j\rangle=\delta_{ij}$$ 
for $m+1\leq i\leq n$ and $1\leq j\leq n$, where $\delta_{ij}$ is the Kronecker symbol. Adding if necessary copies of the $S_j$ to the $\Sigma_i$, 
and tubing around the $T(K_j)$, we can (and we do) assume that the $\Sigma_i$ are embedded in $X=X_0\setminus\sqcup_{m+1\leq j\leq n} Int(T(K_j))$.

Let us compute the integers $b_{ij}$. For $1\leq k\leq m$, let $\ell(K_k)_{ext}$ be a parallel copy of $\ell(K_k)$ in the interior of a regular neighborhood 
of $\partial T(K_k)$ in $X$. We have $\langle \Sigma_i,\ell(K_k)_{ext}\rangle=-b_{ik}$. On the other hand, this algebraic intersection number is equal 
to the linking number $lk(\partial \Sigma_i,\ell(K_k)_{ext})=\sum_{j=1}^m a_{ij} lk(K_j,K_k)=\delta_{ik}$. Hence $b_{ik}=-\delta_{ik}$ and 
$$\partial \Sigma_i=\sum_{j=1}^m a_{ij}\ell(K_j)-m(K_i).$$ 

Now let $N$ be the 3-manifold obtained from $M$ by surgery along the framed link $\sqcup_{1\leq i\leq n} K_i$. The group $H_1(N;\Z)$ is generated by the 
$[m(K_i)]$, which are easily seen to be trivial by considering the surfaces $\Sigma_i$ and $S_i$. Hence $N$ is an integral homology sphere. 
For $1\leq i \leq n$, let $\hat{K}_i\subset N$ be the core of the torus $T(\hat{K}_i)$ reglued during the surgery. Orient each $\hat{K}_i$ so that 
it is homologous to $-m(K_i)$ in $T(\hat{K}_i)$, and parallelize $\hat{K_i}$ with the fixed parallel $\ell(\hat{K}_i)=-m(K_i)$. 
Note that $\ell(K_i)$ is an oriented meridian of $\hat{K}_i$. 
The manifold $M$ is obtained from $N$ by surgery along the framed link $L=\sqcup_{1\leq i\leq n}\hat{K}_i$, 
and the associated generators of $H_1(M;\Z)$ are the $[\ell(K_i)]=\eta_i$. 

Let us compute the linking matrix $A(L)$. 
For $1\leq i\leq m$, construct from $\Sigma_i$ another 2-chain $\hat{\Sigma}_i$ by adding $a_{ij}$ meridian disks in $T(\hat{K}_j)$ for $1\leq j\leq m$, 
and an annulus in $T(\hat{K}_i)$, so that $\partial \hat{\Sigma}_i=\hat{K}_i$. For $1\leq i,j\leq m$, 
we have $lk(\hat{K}_i,\hat{K}_j)=\langle\hat{\Sigma}_i,\ell(\hat{K}_j)\rangle=-a_{ij}$. 
Similarly, we have $lk(\hat{K}_i,\hat{K}_j)=0$ when $1\leq i \leq m$ and $m+1\leq j\leq n$. For $m+1\leq i\leq n$, $m(K_i)$ bounds the surface $-S_i\cap X$ 
which does not meet the $\hat{K}_j$, hence $lk(\hat{K}_i,\hat{K}_j)=0$ for all $1\leq j\leq n$. Finally, $A(L)=A$.
\end{proof}

\fproofof{Theorem \ref{thpres}}
Let $N$ and $\hat{L}\subset N$ by the $\Z$-sphere and the oriented surgery presentation provided by Proposition \ref{propZHS}. 
By \cite[Theorem 2]{Mat}, there is a Y-link $\Gamma\subset S^3$ 
such that $N\cong S^3(\Gamma)$. The genus~3 handlebodies in $N$ reglued during the surgery can be viewed as regular neighborhoods of Y-graphs, 
hence we can assume that the link $\hat{L}$ does not meet them, and we can consider the copy of $\hat{L}$ in $S^3$. Write $\Gamma$ as the disjoint union 
of Y-graphs $\Gamma_i$ for $1\leq i\leq s$, and let $L\subset S^3$ be the framed link defined by the union of all the six-component links 
associated with the $\Gamma_i$'s and of $\hat{L}$. 
\begin{figure}[hbt] 
\begin{center}
\begin{tikzpicture} [scale=0.15]
\newcommand{\bras}[1]{
\draw[rotate=#1] (0,-1.5) circle (2.5);
\draw[rotate=#1,white,line width=8pt] (-0.95,-4) -- (0.95,-4);
\draw[rotate=#1] (0,-11) circle (3) (1,-3.9) -- (1,-7.6);
\draw[rotate=#1,->] (-0.1,-14) -- (0,-14);
\draw[rotate=#1,white,line width=6pt] (-1,-5) -- (-1,-8.7);
\draw[rotate=#1] {(-1,-3.9) -- (-1,-8.7) (-1,-8.7) arc (-180:0:1)};
\draw[rotate=#1,->] (-1,-6) -- (-1,-7);}
\bras{0}
\draw (0,-14) node[below] {$L_{3s+3i}$};
\draw (-1,-6) node[left] {$L_{3i}$};
\draw [white,line width=6pt,rotate=120] (0,-1.5) circle (2.5);
\bras{120}
\draw (18,7.5) node {$L_{3s+3i-2}$};
\draw (9,0) node {$L_{3i-2}$};
\draw [rotate=-120,white,line width=6pt] (-1.77,0.27) arc (135:190:2.5);
\draw [rotate=-120,white,line width=6pt] (1.77,0.27) arc (45:90:2.5);
\bras{-120}
\draw (-18,7.5) node {$L_{3s+3i-1}$};
\draw (-2,6) node {$L_{3i-1}$};
\draw [white,line width=6pt] (-1.77,0.27) arc (135:190:2.5);
\draw [white,line width=6pt] (1.77,0.27) arc (45:90:2.5);
\draw (-1.77,0.27) arc (135:190:2.5);
\draw (1.77,0.27) arc (45:90:2.5);
\end{tikzpicture}
\end{center}
\caption{Six-component link associated with $\Gamma_i$}\label{fignb}
\end{figure}
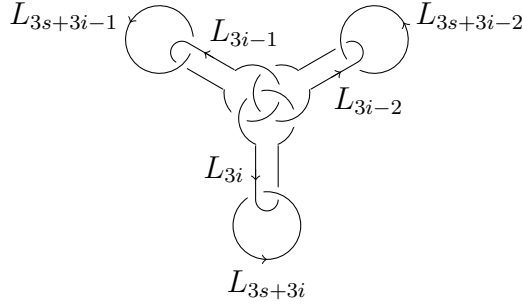
Number and orient the components of $L$ as indicated in Figure \ref{fignb} for the $6s$ first components, 
and set $L_{6s+i}=\hat{L}_i$. The linking matrix of $L$ is:
$$A(L)=\begin{pmatrix} 0 & I_{3s} & 0 \\ I_{3s} & \star & C \\ 0 & \tr{C} & A \end{pmatrix},$$
for some matrix $C$ of size $3s\times n$. The associated presentation of $H_1(M;\Z)$ is $(\b{\gamma},0,..,0,\b{\eta})$, 
where $\b{\gamma}=-C\b{\eta}$. 
Define a congruence matrix $P=\begin{pmatrix} I_{3s} & 0 & -C \\ 0 & I_{3s} & 0 \\ 0 & 0 & I_n \end{pmatrix}$.
Note that $P$ is unimodular. By Lemma \ref{lemmaKirby}, $L$ can be modified by Kirby moves to obtain an oriented surgery presentation $L'$, 
with linking matrix $A(L')=\tr{P}A(L)P=\begin{pmatrix} B & 0 \\ 0 & A \end{pmatrix}$ for some linking matrix $B$ of 
determinant $-1$, and whose associated presentation of $H_1(M;\Z)$ is $\b{\eta}'=(0,..,0,\b{\eta})$. 

Performing if necessary a stabilization, we can assume that $B$ is the matrix of an indefinite odd unimodular bilinear symmetric form over some power of $\Z$, 
and it follows that it is congruent to a diagonal matrix with $\pm1$'s on the diagonal (see for instance \cite[Chap. 2, Th. 4.3]{MilHu}). 
Applying once again Lemma \ref{lemmaKirby}, we obtain an oriented surgery presentation $L''$ whose linking matrix is a stabilization 
of $A$, and whose associated presentation of $H_1(M;\Z)$ is still $\b{\eta}'$.
\fin

    \section{Presentations of $H_1(M;\Z)$, algebraic version} \label{secalg}

In this section, we give an alternative proof of Theorem \ref{thpres}, more technical than the previous one, but also more constructive, 
in the sense that if two 3-manifolds $M$ and $N$ are given by surgery presentations in $S^3$, and if an isomorphism $\xi:H_1(M;\Z)\cong H_1(N;\Z)$ 
preserving the linking pairing is fixed, then the following proof provides Kirby moves which lead to the situation of Proposition \ref{propMuNa}, 
from which one can write down the Y-link realizing $\xi$.
    
Lemma \ref{lemmaKirby} implies that Theorem~\ref{thpres} can be proved by showing that any two $A$-presenta\-tions of a given linked group 
are related, up to stabilizations, by a congruence which induces the required change of generators. We first treat the case of a finite 
linked group. The following proposition is a refinement of \cite[Theorem 4.1]{Dur}.

\begin{proposition} \label{proptech}
 Let $H$ be a finite linked group. Let $A$ and $B$ be non-degenerate linking matrices. Assume there are an $A$-presentation 
$\b{\gamma}=(\gamma_1,..,\gamma_n)$ and a $B$-presentation $\b{\eta}=(\eta_1,..,\eta_m)$ of $H$. Then there are stabilizations 
$\tilde{A}$ of $A$ and $\tilde{B}$ of $B$, and a unimodular matrix $P$, such that $\tilde{A}=\tr{P}\tilde{B}P$ and $P\b{\gamma'}=\b{\eta'}$, 
where $\b{\gamma'}=(0,..,0,\gamma_1,..,\gamma_n)$ and $\b{\eta'}=(0,..,0,\eta_1,..,\eta_m)$.
\end{proposition}
\begin{proof}
 We first fix some formalism. A {\em linked lattice} is a finitely generated free $\Z$-module $R$ equipped with a non-degenerate symmetric 
bilinear form $\psi_R: R\times R \to \Q$. It is {\em integral} if $\psi_R$ takes values in $\Z$. 
Let $V$ be a finite dimensional $\Q$-vector space equipped with a non-degenerate symmetric bilinear 
form $\psi:V\times V\to\Q$. A {\em lattice in V} is a free $\Z$-submodule $R$ of maximal rank. It is linked by the form $\psi$. Define the dual module of $R$ by:
$$R^\sharp=\{x\in V \textrm{ such that } \psi(x,y)\in\Z\textrm{ for all } y\in R\}.$$
The group $R^\sharp$ naturally identifies with $\hom(R;\Z)$ via $x\mapsto(y\mapsto \psi(x,y))$. Note that $(R^\sharp)^\sharp=R$. 
The lattice $R$ is {\em unimodular} if $R^\sharp=R$. 
If the linked lattice $R$ is integral, then $R\subset R^\sharp$, and $\psi$ induces a linking pairing on $R^\sharp/R$. 

Set $V_1=\Q^n$ and denote by $\b{e}=(e_1,..,e_n)$ its canonical basis. Equip $V_1$ with the symmetric bilinear form $\psi_A$ given 
by the matrix $A$ in the basis $\b{e}$. Set $R_1=\Z^n\subset \Q^n$. Let $\b{g}=(g_1,..,g_n)$ be the basis of the dual lattice $R_1^\sharp$ dual to $\b{e}$. 
Note that $g_i$ is given in the basis $\b{e}$ of $V_1$ by the $i$-th column of $A^{-1}$. We have an identification of linked groups $H\cong R_1^\sharp/R_1$ 
given by $\gamma_i\mapsto \bar{g}_i$, where $\bar{g}_i$ is the class of $g_i$ modulo $R_1$. 

Similarly, using the matrix $B$, define a $\Q$-vector space $V_2=\Q^m$ with canonical basis 
$\b{\varepsilon}=(\varepsilon_1,..,\varepsilon_m)$, a symmetric bilinear form $\psi_B$ on $V_2$, a linked lattice 
$R_2=\Z^m\subset V_2$ , a dual lattice $R_2^\sharp$ with a basis $\b{h}=(h_1,..,h_m)$ dual to $\b{\varepsilon}$, 
and an identification of the linked groups $H$ and $R_2^\sharp/R_2$ identifying $\eta_i$ with $\bar{h}_i$.

In $H$, we have $\gamma_i=\sum_{j=1}^m u_{ij}\eta_j$ and $\eta_i=\sum_{j=1}^n v_{ij}\gamma_j$, for some integers $u_{ij}$ and $v_{ij}$. 
Define $f:R_1^\sharp\to R_2^\sharp$ by $g_i\mapsto\sum_{j=1}^m u_{ij}h_j$. The map $f$ induces an isomorphism 
$$\bar{f}:R_1^\sharp/R_1\iso R_2^\sharp/R_2$$ which is the identity on $H$ via the given identifications. 

Set $R=R_1\oplus R_1^\sharp$. Define a linked lattice structure on $R$ by:
$$\psi_R((x,y),(x',y'))=\psi_A(x,y')+\psi_A(x',y)+\psi_A(y,y')-\psi_B(f(y),f(y')).$$ 
The form $\psi_R$ takes values in $\Z$ since $\bar{f}$ is an isomorphism of linked groups. 
The matrix of $\psi_R$ in the basis $(\b{e},\b{g})$ is of the form $\begin{pmatrix} 0 & I \\ I & * \end{pmatrix}$. Hence 
$R$ is a unimodular linked lattice. 

Consider the linked lattice $R\oplus R_2^\sharp$ equipped with the form $\Psi=\psi_R\oplus\psi_B$. 
Define $\iota:R_1^\sharp\to R\oplus R_2^\sharp$ by $y\mapsto(0,y,f(y))$. The map $\iota$ is injective and respects the bilinear forms. 
Set $S=(\iota(R_1^\sharp))^\bot\subset R\oplus R_2^\sharp$. Let us check that $S\subset R\oplus R_2$. Let $(x,y,z)\in S$. 
For $r\in R_1^\sharp$, $\Psi((x,y,z),(0,r,f(r)))=0$ implies $\psi_B(z,f(r))\in\Z$. Now each class of $R_2^\sharp$ modulo $R_2$ 
contains an element $f(r)$, hence $\psi_B(z,s)\in\Z$ for all $s\in R_2^\sharp$, {\em i.e.} $z\in R_2$. 
Hence $S$ equipped with the restriction of $\Psi$ is an integral linked lattice. Define a map 
$\omega:S\oplus R_1^\sharp\to R\oplus R_2^\sharp$ as the direct sum of the inclusion $S\hookrightarrow R\oplus R_2^\sharp$ and of $\iota$.
\begin{lemma}\ 
 \begin{itemize}
  \item The linked lattice $S$ is unimodular.
  \item The map $\omega:S\oplus R_1^\sharp\to R\oplus R_2^\sharp$ is an isomorphism which respects the bilinear forms 
   and identifies $S\oplus R_1$ with $R\oplus R_2$.
 \end{itemize}
\end{lemma}
\begin{proof}
We first prove that $\omega$ is an isomorphism of linked lattices, {\em i.e.} that $R\oplus R_2^\sharp=S\oplus\iota(R_1^\sharp)$. 
Since $A$ is non-degenerate and $S=(\iota(R_1^\sharp))^\bot$, we have $S\cap\iota(R_1^\sharp)=\{0\}$. Hence it suffices to prove 
$R\oplus R_2^\sharp=S+\iota(R_1^\sharp)$. 

We have $R\oplus R_2=S+R_1\subset R\oplus R_2^\sharp$, where $R_1$ is the first component of $R=R_1\oplus R_1^\sharp$. Indeed, for 
$(x,y,z)\in R\oplus R_2=R_1\oplus R_1^\sharp\oplus R_2$, define $\rho:R_1^\sharp\to\Z$ by 
$u\mapsto\psi_A(y,u)-\psi_B(f(y),f(u))+\psi_B(z,f(u))$. Since $\rho\in\hom(R_1^\sharp;\Z)$, there is $w\in R_1$ such that 
$\rho(u)=-\psi_A(w,u)$. It follows that $(w,y,z)\in S$, thus $(x,y,z)\in S+R_1$. 

We have $R\oplus R_2^\sharp=(R\oplus R_2)+\iota(R_1^\sharp)$. Indeed, if $y\in R_2^\sharp$, there is $y_0\in R_2$ such that 
$y=y_0+f(x)$ for some $x\in R_1^\sharp$. Hence $(0,0,y)=(0,-x,y_0)+(0,x,f(x))$.

Finally $R\oplus R_2^\sharp=S+R_1+\iota(R_1^\sharp)$. Now, for $x\in R_1$, $x-\iota(x)\in S$. 
Hence $R\oplus R_2^\sharp=S+\iota(R_1^\sharp)=S\oplus\iota(R_1^\sharp)$.

To show that $S$ is unimodular, we prove that the map $S\to\hom(S;\Z)$ defined by $x\mapsto(y\mapsto\Psi(x,y))$ is bijective. 
Injectivity follows from $R\oplus R_2^\sharp=S\oplus^\bot\iota(R_1^\sharp)$. Let us check surjectivity. 
Let $\phi\in\hom(S;\Z)$. Extend it into $\tilde{\phi}\in\hom(R\oplus R_2^\sharp;\Z)$ by $0$ on $\iota(R_1^\sharp)$. 
Then $\tilde{\phi}=\Psi(v,.)$ for some $v\in(R\oplus R_2^\sharp)^\sharp=R\oplus R_2$. Since $\tilde{\phi}_{|\iota(R_1^\sharp)}=0$, 
$v\in S$ and $\phi=\Psi(v,.)$.

Finally, taking the dual lattices in the equality $R\oplus R_2^\sharp=S\oplus\iota(R_1^\sharp)$, we obtain $R\oplus R_2=S\oplus \iota(R_1)$ 
since $R$ and $S$ are unimodular.
\end{proof}

Back to the proof of the proposition, fix a basis $\b{s}$ of $S$. The form $\Psi$ on $S\oplus R_1\cong R\oplus R_2$ is given by the matrices 
$$\hat{A}=\begin{pmatrix} C & 0 \\ 0 & A \end{pmatrix} \textrm{ in } (\b{s},\iota(\b{e})) \textrm{ and }
\hat{B}=\begin{pmatrix} D & 0 \\ 0 & B \end{pmatrix} \textrm{ in } ((\b{e},\b{g}),\b{\varepsilon}),$$ 
where $C$ and $D$ are unimodular matrices. We have $\tr{P}\hat{B}P=\hat{A}$, where 
$$P=\begin{pmatrix} * & 0 \\ * & A \\ * & F \end{pmatrix},$$ and $F$ is the matrix of $f$ is the bases $\b{e}$ and $\b{\varepsilon}$. 
Let us prove that $P\b{\gamma'}=\b{\eta'}$. 

Set $U=(u_{ij})_{\substack{1\leq i\leq n \\ 1\leq j\leq m}}$ and $V=(v_{ij})_{\substack{1\leq i\leq m \\ 1\leq j\leq n}}$. 
The matrix of $f$ in the bases $\b{g}$ and $\b{h}$ is~$\tr{U}$. Hence $F=B^{-1}\tr{U}A$. 
Moreover, since $\bar{f}$ is an isomorphism of linked groups, we have $A^{-1}=UB^{-1}\tr{U}\ mod\ \M_n(\Z)$. 

For $1\leq i\leq n$, $\eta_i=\sum_{k=1}^m(VU)_{ik}\eta_k$ in $H$. Hence $\sum_{k=1}^m(VU)_{ik}h_k-h_i\,\in R_2$. 
Thus $VU=I\ mod\ \M_m(\Z)B$. Finally:
\begin{eqnarray*}
 V&=&VUB^{-1}\tr{U}A\ mod\ \M_{m,n}(\Z)A \\ &=& B^{-1}\tr{U}A\ mod\ \M_{m,n}(\Z)A \\ &=& F \ mod\ \M_{m,n}(\Z)A.
\end{eqnarray*}
Hence there is a matrix $G$ such that $V=F+GA$, and
$$P\b{\gamma'}=\begin{pmatrix} * & 0 \\ * & A \\ * & V-GA \end{pmatrix} \begin{pmatrix} 0 \\ \b{\gamma} \end{pmatrix}=
\begin{pmatrix} 0 \\ A\b{\gamma} \\ (V-GA)\b{\gamma} \end{pmatrix}=\begin{pmatrix} 0 \\ 0 \\ \b{\eta} \end{pmatrix} \textrm{ in } H^{2n+m}.$$

It remains to modify $\hat{A}$ and $\hat{B}$ into stabilizations of $A$ and $B$ without acting on the presentations $\b{\gamma'}$ 
and $\b{\eta'}$. 

For a unimodular linked lattice $Z$ with form $\psi_Z$, it is known that if $\psi_Z$ is indefinite (there is $z\in Z$ 
such that $\psi_Z(z,z)=0$) and odd (there is $z\in Z$ such that $\psi_Z(z,z)$ is odd), then there is a basis of $Z$ in which 
the matrix of $\psi_Z$ is diagonal (see for instance \cite[Chap. 2, Th. 4.3]{MilHu}). Making if necessary a direct sum of $R$ and $S$ with 
a copy of $\Z$ equipped with a form given by $1$ or $-1$ on a generator, we can assume that $R$ and $S$ are equipped with an odd indefinite 
form. Hence the matrices $C$ and $D$ are congruent to diagonal matrices with $\pm1$ on the diagonal. Applying these 
congruences on $\hat{A}$ and $\hat{B}$, we obtain stabilizations $\tilde{A}$ and $\tilde{B}$ of $A$ and $B$ related by a congruence 
such that the congruence matrix acts on $\b{\gamma'}$ as $P$.
\end{proof}

\sproofof{Theorem \ref{thpres}}
Thanks to Lemma \ref{lemmaKirby}, and recalling that any linking matrix is congruent to an admissible one, we may assume that the linking matrix $A$ is admissible. 
Write $A=\begin{pmatrix} B & 0 \\ 0 & 0 \end{pmatrix}$, where $B$ is a non-degenerate linking matrix. 
Let $L$ be an oriented surgery presentation of $M$ such that 
$A(L)=\begin{pmatrix} B_L & 0 \\ 0 & 0 \end{pmatrix}$ where $B_L$ is a non-degenerate linking matrix. Let $\b{m}$ be the associated 
presentation of $H_1(M;\Z)$. 
Let $\beta$ be the rank of $H_1(M;\Z)$, and let $k$ (resp. $\ell$) be the size of $B$ (resp. $B_L$). 
We shall modify the presentation $\b{\eta}$ so that the $\beta$ last generators coincide with those of $\b{m}$. 
There are matrices $C$ and $D$, with integral coefficients, such that $\det(D)=\pm1$, and 
$$\begin{pmatrix} m_{\ell+1} \\ \vdots \\ m_{\ell+\beta} \end{pmatrix}=C\begin{pmatrix} \eta_1 \\ \vdots \\ \eta_k \end{pmatrix}+
D\begin{pmatrix} \eta_{k+1} \\ \vdots \\ \eta_{k+\beta} \end{pmatrix}.$$
Set $P=\begin{pmatrix} I & 0 \\ C & D \end{pmatrix}$. We have $\tr{P}AP=A$ and 
$P\eta=\begin{pmatrix} \eta_1 \\ \vdots \\ \eta_k \\ m_{\ell+1} \\ \vdots \\ m_{\ell+\beta} \end{pmatrix}$. 

By Proposition \ref{proptech}, there are stabilizations $\tilde{B}$ of $B$ and $\tilde{B}_L$ of $B_L$, and a unimodular matrix $Q$, 
such that $\tr{Q}\tilde{B}_LQ=\tilde{B}$ and $Q\begin{pmatrix} 0 \\ \vdots \\ 0 \\ \eta_1 \\ \vdots \\ \eta_k \end{pmatrix}
=\begin{pmatrix} 0 \\ \vdots \\ 0 \\ m_1 \\ \vdots \\ m_\ell \end{pmatrix}.$

\newcommand{\ALt}{A\hspace{-1ex}\widetilde{\hspace{1ex}(L})}
Set $\tilde{P}=\begin{pmatrix} Q & 0 \\ 0 & I \end{pmatrix}\begin{pmatrix} I & 0 \\ 0 & P \end{pmatrix}$ 
(note that the size of the blocs is different for the two matrices). 
Set $\ALt=\begin{pmatrix} \tilde{B}_L & 0 \\ 0 & 0 \end{pmatrix}$. Then $\tilde{A}=\tr{\tilde{P}}\ALt\tilde{P}$ 
is a stabilization of $A$ and $\b{\eta'}=(0,..,0,\eta_1,..,\eta_{k+\beta})$ is an $\tilde{A}$-presentation of $H_1(M;\Z)$ 
such that $\tilde{P}\b{\eta'}=\b{m'}=(0,..,0,m_1,..,m_{\ell+\beta})$. Up to stabilization, we can assume that $\det(\tilde{P})=1$. 

Perform stabilizations of $L$ in order to obtain a surgery link with linking matrix $\ALt$. Then apply Lemma \ref{lemmaKirby} 
to obtain a surgery link with linking matrix $\tilde{A}$ and associated presentation $\b{\eta'}$ of $H_1(M;\Z)$.
\fin

    \section{Example} \label{secex}

We give in this section an example of a borromean surgery which realizes a non-trivial isomorphism of the first homology group 
of a given 3-manifold, namely the multiplication by $-1$. The method can be applied to any surgery presentation; we consider here 
a presentation given by a non-invertible knot in order to make things more explicit.
    
The knot $9_{32}$ drawn in Figure \ref{fig9_32} is a non-invertible knot, 
\begin{figure}[htb] 
\begin{center}
\begin{tikzpicture} [scale=0.17]
\draw[thick] 
(30,10) .. controls +(-1,-1) and +(1,-1) .. (24,10) -- (22,12) .. controls +(-1,1) and +(1,1) .. (20,12) 
(18,10) .. controls + (-1,-1) and +(2,0) .. (15,7) (13,7) -- (5,7) (3,7) .. controls +(-1,0) and +(0,1) .. (1,6) 
(1,4) -- (1,2) .. controls +(0,-2) and +(0,-2) .. (8,2) (8,4) -- (8,6) (8,8) -- (8,10) 
(8,12) -- (8,17) .. controls +(0,5) and +(-6,0) .. (2,16) -- (4,16) (6,16) -- (7,16) (9,16) -- (11,16) 
(13,16) -- (15,16) .. controls +(2,0) and +(-1,1) .. (18,12) -- (20,10) .. controls +(1,-1) and +(-1,-1) .. (22,10) 
(24,12) .. controls +(1,1) and +(-1,0) .. (25,14) -- (27,14) (29,14) .. controls +(1,0) and +(-1,1) .. (30,12) -- (32,10)
.. controls +(1,-1) and +(-1,-1) .. (41,0) -- (43,2) .. controls +(1,1) and +(0,-1) .. (46,4) -- (46,6) 
(46,8) .. controls +(0,1) and +(-1,-1) .. (48,9) (50,11) .. controls +(1,1) and +(1,-1) .. (47,18) -- (45,20) 
.. controls +(-1,1) and +(-1,-1) .. (45,22) (47,24) .. controls +(1,1) and +(-1,1) .. (57,22) -- (59,20) 
.. controls +(1,-1) and +(-1,0) .. (60,19) -- (62,19) (64,19) -- (66,19) (68,19) -- (70,19) 
(72,19) -- (74,19) .. controls +(10,0) and +(0,12) .. (38,24) -- (38,22) .. controls +(0,-1) and +(-1,1) .. (39,17) 
(41,15) .. controls +(1,-1) and +(-1,0) .. (43,13) (45,13) -- (47,13) .. controls +(1,0) and +(-1,1) .. (48,11) -- (50,9) 
.. controls +(1,-1) and +(-1,0) .. (56,8) (58,8) .. controls +(2,0) and +(0,-2) .. (67,11) 
(67,13) -- (67,20) .. controls +(0,2) and +(1,1) .. (59,22) (57,20) .. controls +(-0.5,-0.5) and +(0,1) .. (57,15) 
(57,13) -- (57,7) .. controls +(0,-0.5) and +(0.3,0.6) .. (56.3,5.1) (55.7,3.9) -- (55.3,3.1) 
(54.7,1.9) .. controls +(-1,-2) and +(1,-1) .. (43,0) (41,2) .. controls +(-1,1) and +(0,-2) .. (38,10) 
(38,12) -- (38,14) .. controls +(0,1) and +(-0.5,-0.5) .. (39,15) -- (41,17) .. controls +(1,1) and +( -1,-1) .. (45,18) 
(47,20) .. controls +(1,1) and +(1,-1) .. (47,22) -- (45,24) .. controls +(-1,1) and +(2,0) .. (39,23) 
(37,23) .. controls +(-2,0) and +( 1,1) .. (32,12);
\draw[thick,->] (30,12) -- (32,10);
\draw[thick] (32,6) node {$K$};
\draw[blue] 
(2.7,4) -- (6.3,4) (4,4) -- (4,14.3) 
(0,3) .. controls +(-1,0) and +(-1,0) .. (0,5) -- (2,5) .. controls +(1,0) and +(1,0) .. (2,3)
(7,5) .. controls +(-1,0) and +(-1,0) .. (7,3) -- (9,3) .. controls +(1,0) and +(1,0) .. (9,5)
(3,15) .. controls +(0,-1) and +(0,-1) .. (5,15) -- (5,17) .. controls +(0,1) and +(0,1) .. (3,17)
(9.7,12) -- (13,12) (13,8.7) -- (13,14.3)
(7,13) .. controls +(-1,0) and +(-1,0) .. (7,11) -- (9,11) .. controls +(1,0) and +(1,0) .. (9,13)
(12,6) .. controls +(0,-1) and +(0,-1) .. (14,6) -- (14,8) .. controls +(0,1) and +(0,1) .. (12,8)
(14,15) .. controls +(0,-1) and +(0,-1) .. (12,15) -- (12,17) .. controls +(0,5) and +(0,5) .. (28,15) -- (28,13) 
.. controls +(0,-1) and +(0,-1) .. (26,13) (26,15) .. controls +(0,3) and +(0,3) .. (14,17)
(58.7,13) -- (65.3,13) (62,13) -- (62,17.3) 
(56,12) .. controls +(-1,0) and +(-1,0) .. (56,14) -- (58,14) .. controls +(1,0) and +(1,0) .. (58,12)
(66,14) .. controls +(-1,0) and +(-1,0) .. (66,12) -- (68,12) .. controls +(1,0) and +(1,0) .. (68,14)
(61,18) .. controls +(0,-1) and +(0,-1) .. (63,18) -- (63,20) .. controls +(0,1) and +(0,1) .. (61,20)
(39.7,12) -- (42,9) -- (44.3,6) (42,9) -- (45,11.3) 
(37,13) .. controls +(-1,0) and +(-1,0) .. (37,11) -- (39,11) .. controls +(1,0) and +(1,0) .. (39,13)
(46,12) .. controls +(0,-1) and +(0,-1) .. (44,12) -- (44,14) .. controls +(0,1) and +(0,1) .. (46,14)
(45,5) .. controls +(-1,0) and +(-1,0) .. (45,7) -- (47,7) .. controls +(2,0) and +(-2,1) .. (55,5) -- (57,4) 
.. controls +(8,-4) and +(0,-10) .. (71,18) -- (71,20) .. controls +(0,1) and +(0,1) .. (73,20) 
(73,18) .. controls +(0,-12) and +(10,-5) .. (56,2) -- (54,3) .. controls +(-2,1) and +(1,0) .. (47,5);
\draw[blue] (70,3) node {$\Gamma$};
\end{tikzpicture}
\caption{A Y-link inverting the knot $9_{32}$} \label{fig9_32}
\end{center}
\end{figure}
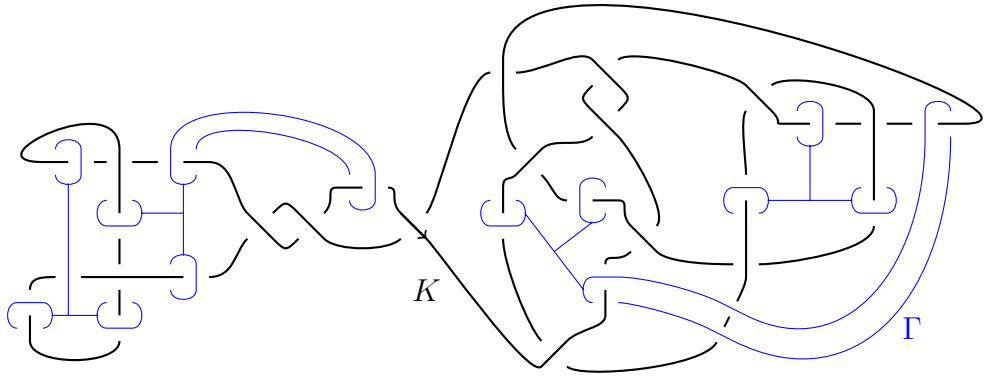
{\em i.e.} the two possible orientations provide non-isotopic oriented knots. 
Let $9_{32}^+$ be the one represented in Figure \ref{fig9_32}, and let $9_{32}^-$ be the other one. 
Let $K\subset S^3$ be the knot $9_{32}^+$ with framing $n>1$. Any knot can be obtain from the unknot with the same framing by borromean surgeries, 
and conversely (see the proof of Proposition \ref{propMuNa}). 
Figure \ref{fig9_32} shows the knot $K$ as the connected sum of an unknot and a $9_{32}^+$. The borromean surgery on the drawn Y-link $\Gamma$ 
modifies the right part by trivializing the $9_{32}^+$, and the left part by transforming the unknot into a $9_{32}^-$. 
Hence it globally changes the $9_{32}^+$ into a $9_{32}^-$. 

Set $M=S^3(K)$. Note that $H_1(M;\Z)\cong\Z/n\Z$ is non-trivial and finite. 
The manifolds $M$ and $M(\Gamma)$ are both obtained from $S^3$ by surgery on a knot $9_{32}$, hence they are homeomorphic. 
However, the oriented meridians of the oriented surgery presentations given by $9_{32}^+$ and $9_{32}^-$ 
have opposite homology classes. Hence the Y-link $\Gamma$ does not modify the manifold $M$, 
up to homeomorphism, but it realizes the non-trivial isomorphism of $H_1(M;\Z)$ given by multiplication by $-1$.

\def\cprime{$'$}
\providecommand{\bysame}{\leavevmode ---\ }
\providecommand{\og}{``}
\providecommand{\fg}{''}
\providecommand{\smfandname}{\&}
\providecommand{\smfedsname}{\'eds.}
\providecommand{\smfedname}{\'ed.}
\providecommand{\smfmastersthesisname}{M\'emoire}
\providecommand{\smfphdthesisname}{Th\`ese}

\end{document}